\UseRawInputEncoding
\documentclass[12pt]{article}

\usepackage{amsfonts,amsmath,amsthm}
\usepackage{amssymb}
\usepackage{xcolor}
\usepackage{graphicx}
\usepackage{stmaryrd}
\usepackage{dsfont}

\usepackage[paper=a4paper,dvips,top=2cm,left=2cm,right=2cm,
foot=1cm,bottom=4cm]{geometry}

\newcommand{\vx}{{\bf x}}
\newcommand{\vy}{{\bf y}}

\newcommand{\R}{\mathbb{R}}

\newtheorem{Thm}{Theorem}[section]

\newtheorem{Lem}[Thm]{Lemma}
\newtheorem{Prop}[Thm]{Proposition}
\newtheorem{Cor}[Thm]{Corollary}
\newtheorem{example}[Thm]{Example}
\newtheorem{remark}[Thm]{Remark}
\newtheorem{conjecture}[Thm]{Conjecture}

\begin{document}

\title{On the SOS Rank of Simple and Diagonal Biquadratic Forms}
%\Large
\author{Yi Xu\footnote{School of Mathematics, Southeast University, Nanjing  211189, China. Nanjing Center for Applied Mathematics, Nanjing 211135,  China. Jiangsu Provincial Scientific Research Center of Applied Mathematics, Nanjing 211189, China. ({\tt yi.xu1983@hotmail.com})}
		\and
		Chunfeng Cui\footnote{School of Mathematical Sciences, Beihang University, Beijing  100191, China.
			({\tt chunfengcui@buaa.edu.cn})}
		\and {and \
			Liqun Qi\footnote{%Jiangsu Provincial Scientific Research Center of Applied Mathematics, Nanjing 211189, China.
				Department of Applied Mathematics, The Hong Kong Polytechnic University, Hung Hom, Kowloon, Hong Kong.
				({\tt maqilq@polyu.edu.hk})}
		}
	}

\date{\today}
\maketitle

\begin{abstract}
We study the sum-of-squares (SOS) rank of simple and diagonal biquadratic forms.
For simple biquadratic forms in $3 \times 3$ variables, we show that the maximum SOS rank is exactly $6$, attained by a specific six-term form.
We further prove that for any $m \ge 3$, there exists an $m \times m$ simple biquadratic form whose SOS rank is exactly $2m$.
Moreover, we show that for all $m, n \ge 3$, the maximum SOS rank over $m \times n$ simple biquadratic forms is at least $m+n$, which implies $\mathrm{BSR}(m,n) \ge m+n$.
For diagonal biquadratic forms with nonnegative coefficients, we prove an SOS rank upper bound of $7$, improving the general bound of $8$ for $3 \times 3$ forms.
These results provide new lower and upper bounds on the worst-case SOS rank of biquadratic forms and highlight the role of structure in reducing the required number of squares.

\medskip

\textbf{Keywords.} Biquadratic forms, sum-of-squares, SOS rank, simple forms, diagonal forms, positive semidefinite.

\medskip
\textbf{AMS subject classifications.} 11E25, 12D15, 14P10, 15A69, 90C23.
\end{abstract}

\section{Introduction}
Let $m, n \ge 2$.
A biquadratic form in variables $\vx = (x_1,\dots,x_m)$ and $\vy = (y_1,\dots,y_n)$ is a homogeneous polynomial
\[
P(\vx,\vy) = \sum_{i,k=1}^{m} \sum_{j,l=1}^{n} a_{ijkl} x_i x_k y_j y_l,
\]
with real coefficients $a_{ijkl}$.
It is called \emph{positive semidefinite (PSD)} if $P(\vx,\vy) \ge 0$ for all $\vx,\vy$, and \emph{sum-of-squares (SOS)} if it can be written as a finite sum of squares of bilinear forms.
The smallest number of squares required is the \emph{SOS rank} of $P$, denoted $\operatorname{sos}(P)$.

Understanding the SOS rank of structured biquadratic forms is a fundamental problem in real algebraic geometry and polynomial optimization.
In this note, we focus on two natural subclasses: \emph{simple} biquadratic forms (containing only distinct terms $x_i^2y_j^2$) and \emph{diagonal} biquadratic forms (terms $a_{ij}x_i^2y_j^2$).
Both classes are automatically SOS when PSD, but their minimal SOS ranks are nontrivial.

\begin{itemize}
    \item For $3 \times 3$ simple biquadratic forms, the maximum SOS rank is exactly $6$ (Theorem~\ref{thm:max_simple}).
    \item For any $m \ge 3$, there exists an $m \times m$ simple biquadratic form with SOS rank exactly $2m$, generalizing the $3\times3$ case (Theorem~\ref{thm:m-by-m}).
    \item { For all $m, n \ge 3$, the maximum SOS rank of $m \times n$ simple biquadratic forms is at least $m+n$, implying $\mathrm{BSR}(m,n) \ge m+n$ (Theorem~\ref{thm:mn-lower-bound}).}
    \item For $3 \times 3$ diagonal biquadratic forms, the SOS rank never exceeds $7$, improving the general bound $mn-1=8$ (Theorem~\ref{thm:diagonal_bound}).
\end{itemize}
Beyond structured forms, our work also provides evidence for the exact value of the maximum SOS rank among all $3\times3$ biquadratic forms.
We conjecture that $\mathrm{BSR}(3,3)=6$ (Conjecture~\ref{conj:BSR33}), a value attained by the simple form $P_{3,3,6}$;
we support this by showing that enriching this form with certain additional squares does not increase its SOS rank.

\section{Simple Biquadratic Forms}
Let $m \ge n$.
A biquadratic form is called \textbf{simple} if it contains only distinct terms of the type $x_i^2 y_j^2$.
We define a family of simple forms $P_{m,n,s}$ where $s = 1,\dots,mn$ counts the number of square terms, following a fixed ordering of index pairs $(i,j)$ (see \cite{QCX26} for details).
For $m=n=3$ the first six forms are:
\[
\begin{aligned}
P_{3,3,1} &= x_1^2 y_1^2, \\
P_{3,3,2} &= x_1^2 y_1^2 + x_2^2 y_2^2, \\
P_{3,3,3} &= x_1^2 y_1^2 + x_2^2 y_2^2 + x_3^2 y_3^2, \\
P_{3,3,4} &= x_1^2 y_1^2 + x_2^2 y_2^2 + x_3^2 y_3^2 + x_1^2 y_2^2, \\
P_{3,3,5} &= x_1^2 y_1^2 + x_2^2 y_2^2 + x_3^2 y_3^2 + x_1^2 y_2^2 + x_2^2 y_3^2, \\
P_{3,3,6} &= x_1^2 y_1^2 + x_2^2 y_2^2 + x_3^2 y_3^2 + x_1^2 y_2^2 + x_2^2 y_3^2 + x_3^2 y_1^2 .
\end{aligned}
\]

\begin{Thm}[A $3 \times 3$ simple form requiring six squares]\label{thm:3x3-tight}
The form
\[
P'(\vx,\vy) = P_{3,3,6}(\vx,\vy) \] \[ = x_1^2 y_1^2 + x_2^2 y_2^2 + x_3^2 y_3^2 + x_1^2 y_2^2 + x_2^2 y_3^2 + x_3^2 y_1^2
\]
satisfies $\operatorname{sos}(P') = 6$.
\end{Thm}
\begin{proof}
The upper bound $\operatorname{sos}(P') \le 6$ is immediate from the decomposition
\[
P' = (x_1y_1)^2 + (x_2y_2)^2 + (x_3y_3)^2\] \[ + (x_1y_2)^2 + (x_2y_3)^2 + (x_3y_1)^2 .
\]
For the lower bound, assume a representation with only five squares.
Writing each square as a bilinear form and comparing coefficients leads to six pairwise orthogonal unit vectors in $\R^5$, which is impossible.
Hence $\operatorname{sos}(P') \ge 6$; see \cite{QCX26} for details.
\end{proof}

To bound the SOS rank of simple forms with more terms, we use a combinatorial observation.
For $i\neq j$ and $k\neq l$ define
\[
T_{ijkl}(\vx,\vy) = x_i^2y_k^2 + x_j^2y_l^2 + x_i^2y_l^2 + x_j^2y_k^2\] \[ = (x_iy_k + x_jy_l)^2 + (x_iy_l - x_jy_k)^2,
\]
so $\operatorname{sos}(T_{ijkl}) = 2$.

\begin{Lem}\label{lem:contains_Tij}
Every simple biquadratic form in variables $(x_1,x_2,x_3)$ and $(y_1,y_2,y_3)$ containing at least $7$ distinct terms $x_a^2y_b^2$ contains $T_{ijkl}$ for some $i\neq j$, $k\neq l$.
\end{Lem}
\begin{proof}
The $3\times3$ grid of possible terms has $9$ cells.
Choosing $7$ cells leaves at most $2$ empty.
Any rectangle of four cells $(i,k),(i,l),(j,k),(j,l)$ must be completely filled, otherwise the two empty cells would have to avoid all four corners of that rectangle, which cannot happen for all possible rectangles simultaneously.
\end{proof}

\begin{Prop}\label{prop:8term_SOSrank}
Any eight-term simple biquadratic form has $\operatorname{sos}(P) \le 6$.
Any seven-term simple biquadratic form has $\operatorname{sos}(P) \le 5$.
\end{Prop}
\begin{proof}
By Lemma~\ref{lem:contains_Tij}, such a form contains some $T_{ijkl}$.
Write $P = T_{ijkl} + Q$, where $Q$ consists of the remaining terms.
Since $T_{ijkl}$ needs two squares and each term of $Q$ is a single square, we obtain the stated bounds.
\end{proof}

\begin{Prop}\label{prop:9term_SOSrank}
The nine-term form $P_{3,3,9} = \sum_{i=1}^3 \sum_{j=1}^3 x_i^2 y_j^2$ satisfies $\operatorname{sos}(P_{3,3,9}) \le 4$.
\end{Prop}
\begin{proof}
The following explicit decomposition uses only four squares:
\[
P_{3,3,9} = (x_1 y_1 + x_2 y_2 + x_3 y_3)^2 + (x_2 y_3 - x_3 y_2)^2\] \[ + (x_3 y_1 - x_1 y_3)^2 + (x_1 y_2 - x_2 y_1)^2.
\]
Expanding verifies that all cross terms cancel.
\end{proof}

\begin{Thm}\label{thm:max_simple}
The maximum SOS rank of $3 \times 3$ simple biquadratic forms is $6$.
\end{Thm}
\begin{proof}
The six-term form $P_{3,3,6}$ has SOS rank $6$ (Theorem~\ref{thm:3x3-tight}).
For any simple form with $t$ terms, Proposition~\ref{prop:8term_SOSrank} and Proposition~\ref{prop:9term_SOSrank} show $\operatorname{sos}(P) \le 6$ when $t \ge 7$, while for $t \le 6$ trivially $\operatorname{sos}(P) \le t \le 6$.
Hence $6$ is the maximum.
\end{proof}

\subsection{Extension to $m \times m$ Simple Forms}
The $3\times3$ result can be generalized to higher dimensions, yielding a linear lower bound on the SOS rank of simple biquadratic forms.

\begin{Lem}\label{lem:support-property}
For $m\ge 3$, let
\[
S = \{(i,i):i=1,\dots,m\} \cup \{(i,i+1):i=1,\dots,m\}
\]
with indices taken modulo $m$. Then for any two distinct pairs $(i,j),(p,q)\in S$ with $i\neq p$ and $j\neq q$, at least one of $(i,q)$ or $(p,j)$ does not belong to $S$.
\end{Lem}
\begin{proof}
Suppose for contradiction that both $(i,q)\in S$ and $(p,j)\in S$.
Since $(i,j)\in S$, we have $j\in\{i,i+1\}$; since $(i,q)\in S$, we have $q\in\{i,i+1\}$.
Similarly, from $(p,q)\in S$ we get $q\in\{p,p+1\}$, and from $(p,j)\in S$ we get $j\in\{p,p+1\}$.
Thus $\{j,q\}\subseteq\{i,i+1\}\cap\{p,p+1\}$. As $j\neq q$, the intersection contains two distinct elements, which forces $\{i,i+1\}=\{p,p+1\}$ as sets modulo $m$.
For $m\ge 3$, this implies $p\equiv i$ (mod~$m$), contradicting $i\neq p$.
\end{proof}

\begin{Thm}\label{thm:m-by-m}
Let $m \ge 3$.
Define the $m \times m$ simple biquadratic form
\[
P_{m,m,2m}(\vx,\vy) = \sum_{i=1}^m x_i^2 y_i^2 + \sum_{i=1}^m x_i^2 y_{i+1}^2,
\]
where the index $y_{m+1}$ is interpreted as $y_1$ (cyclic addition modulo $m$).
Then
\[
\operatorname{sos}\big(P_{m,m,2m}\big) = 2m.
\]
\end{Thm}
\begin{proof}
The upper bound $\operatorname{sos}(P_{m,m,2m}) \le 2m$ is obvious: each square term $x_i^2y_j^2$ can be taken as a single square $(\sqrt{1}\,x_i y_j)^2$, giving a decomposition with exactly $2m$ squares.

For the lower bound, assume that $P_{m,m,2m}$ admits an SOS decomposition with $R$ squares, i.e.
\[
P_{m,m,2m} = \sum_{t=1}^R L_t(\vx,\vy)^2,
\]
where each $L_t$ is bilinear:
\[
L_t(\vx,\vy) = \sum_{i=1}^m \sum_{j=1}^m c_{ij}^{(t)} x_i y_j, \qquad t=1,\dots,R.
\]

Because $P_{m,m,2m}$ contains only the $2m$ square terms listed and no cross terms $x_i x_p y_j y_q$ with $(i,j)\neq (p,q)$, we have for every $(i,j)$ in the support $S$:
\[
\sum_{t=1}^R \bigl(c_{ij}^{(t)}\bigr)^2 = 1,
\]
and for $(i,j)\notin S$, $c_{ij}^{(t)} = 0$ for all $t$.

Now consider the cross terms $x_i x_p y_j y_q$ with $(i,j) \neq (p,q)$.
Since such terms do not appear in $P_{m,m,2m}$, their coefficients in the expansion $\sum_{t=1}^R L_t^2$ must vanish.
For $i=p$ or $j=q$, this yields respectively
\[
\sum_{t=1}^R c_{ij}^{(t)} c_{iq}^{(t)} = 0 \quad (j\neq q), \qquad
\sum_{t=1}^R c_{ij}^{(t)} c_{pj}^{(t)} = 0 \quad (i\neq p).
\]
For $i\neq p$ and $j\neq q$, the coefficient is
\[
\sum_{t=1}^R \bigl( c_{ij}^{(t)} c_{pq}^{(t)} + c_{iq}^{(t)} c_{pj}^{(t)} \bigr) = 0.
\]

By Lemma~\ref{lem:support-property}, for distinct $(i,j),(p,q)\in S$ with $i\neq p$ and $j\neq q$, either $c_{iq}^{(t)}=0$ for all $t$ or $c_{pj}^{(t)}=0$ for all $t$.
Thus the condition $\sum_t (c_{ij}^{(t)}c_{pq}^{(t)} + c_{iq}^{(t)}c_{pj}^{(t)})=0$ reduces to $\sum_t c_{ij}^{(t)}c_{pq}^{(t)}=0$.

Combined with the conditions for $i=p$ or $j=q$, we conclude that for \emph{any} two distinct pairs $(i,j),(p,q)\in S$,
\[
\sum_{t=1}^R c_{ij}^{(t)} c_{pq}^{(t)} = 0.
\]

Define the vectors $\mathbf{C}_{ij} = \bigl(c_{ij}^{(1)},\dots,c_{ij}^{(R)}\bigr) \in \R^R$ for $(i,j)\in S$.
The conditions $\sum_t (c_{ij}^{(t)})^2 = 1$ and $\sum_t c_{ij}^{(t)} c_{pq}^{(t)} = 0$ for $(i,j)\neq (p,q)$ imply that the $2m$ vectors $\{\mathbf{C}_{ij} : (i,j)\in S\}$ are pairwise orthogonal and each has Euclidean norm $1$.
Consequently, they form a set of $2m$ orthonormal vectors in $\R^R$, which is possible only if $R \ge 2m$.
Thus $\operatorname{sos}(P_{m,m,2m}) \ge 2m$.

Combining the two inequalities gives $\operatorname{sos}(P_{m,m,2m}) = 2m$.
\end{proof}

{
\subsection{A General Lower Bound for $\mathrm{BSR}_{\mathrm{simple}}(m,n)$}

We denote by \(\mathrm{BSR}_{\mathrm{simple}}(m,n)\) the maximum SOS rank over all \emph{simple} \(m \times n\) biquadratic forms.

\begin{Lem}[Monotonicity]\label{lem:monotonicity}
For any $m, n \ge 3$, we have $\mathrm{BSR}_{\mathrm{simple}}(m, n+1) \ge \mathrm{BSR}_{\mathrm{simple}}(m, n) + 1$.
\end{Lem}
\begin{proof}
Let $P(\vx, \vy)$ be an $m \times n$ simple biquadratic form with $\operatorname{sos}(P) = \mathrm{BSR}_{\mathrm{simple}}(m, n)$. Consider the $m \times (n+1)$ simple biquadratic form
\[
P_+(\vx, \vy, y_{n+1}) = P(\vx, \vy) + x_1^2 y_{n+1}^2.
\]
Clearly, $\operatorname{sos}(P_+) \le \operatorname{sos}(P) + 1 = \mathrm{BSR}_{\mathrm{simple}}(m, n) + 1$.

We now show that $\operatorname{sos}(P_+) \ge \mathrm{BSR}_{\mathrm{simple}}(m, n) + 1$. Suppose, for contradiction, that $P_+$ admits an SOS decomposition with $R = \mathrm{BSR}_{\mathrm{simple}}(m, n)$ squares:
\[
P_+ = \sum_{t=1}^R L_t(\vx, \vy, y_{n+1})^2,
\]
where each $L_t$ is bilinear in $\vx$ and $(\vy, y_{n+1})$. Write $L_t = M_t(\vx, \vy) + N_t(\vx) y_{n+1}$, where $M_t$ is bilinear in $\vx, \vy$ and $N_t$ is linear in $\vx$. Expanding and comparing coefficients, we have
\[
\sum_{t=1}^R M_t(\vx, \vy)^2 = P(\vx, \vy), \qquad \sum_{t=1}^R N_t(\vx)^2 = x_1^2, \qquad \sum_{t=1}^R M_t(\vx, \vy) N_t(\vx) = 0.
\]
The first equality shows that the $M_t$ give an SOS representation of $P$ with $R$ squares. Since $R = \operatorname{sos}(P)$ is minimal, the $M_t$ are linearly independent as bilinear forms. The second equality implies that the $N_t$ are multiples of $x_1$: $N_t = c_t x_1$ with $\sum_t c_t^2 = 1$. Then the third equality becomes $\sum_t c_t M_t(\vx, \vy) = 0$, which is a linear dependence among the $M_t$, a contradiction. Hence, $\operatorname{sos}(P_+) \ge R+1$, and so $\operatorname{sos}(P_+) = \mathrm{BSR}_{\mathrm{simple}}(m, n) + 1$. Therefore, $\mathrm{BSR}_{\mathrm{simple}}(m, n+1) \ge \mathrm{BSR}_{\mathrm{simple}}(m, n) + 1$.
\end{proof}

By symmetry, the same argument applied to the variables $\vx$ and $\vy$ yields the following corollary.

\begin{Cor}[Monotonicity in the first dimension]\label{cor:monotonicity-first}
For any $m, n \ge 3$, we have $\mathrm{BSR}_{\mathrm{simple}}(m+1, n) \ge \mathrm{BSR}_{\mathrm{simple}}(m, n) + 1$.
\end{Cor}

\begin{Thm}\label{thm:mn-lower-bound}
For all $m, n \ge 3$, we have $\mathrm{BSR}_{\mathrm{simple}}(m, n) \ge m + n$. Consequently, $\mathrm{BSR}(m, n) \ge m + n$.
\end{Thm}
\begin{proof}
We prove the inequality by induction on $k = m + n$.

\noindent\textbf{Base case:} $k = 6$, i.e., $m = n = 3$. By Theorem~\ref{thm:max_simple}, $\mathrm{BSR}_{\mathrm{simple}}(3,3)=6 = 3+3$.

\noindent\textbf{Inductive step:} Assume the inequality holds for all $m', n' \ge 3$ with $m' + n' < k$. Consider $m, n \ge 3$ with $m + n = k$.

If $m > 3$, then by Corollary~\ref{cor:monotonicity-first},
\[
\mathrm{BSR}_{\mathrm{simple}}(m, n) \ge \mathrm{BSR}_{\mathrm{simple}}(m-1, n) + 1.
\]
Since $(m-1) + n = k - 1$, the induction hypothesis gives $\mathrm{BSR}_{\mathrm{simple}}(m-1, n) \ge (m-1) + n = m + n - 1$. Hence,
\[
\mathrm{BSR}_{\mathrm{simple}}(m, n) \ge (m + n - 1) + 1 = m + n.
\]

If $m = 3$, then $n = k - 3 \ge 3$. By Lemma~\ref{lem:monotonicity},
\[
\mathrm{BSR}_{\mathrm{simple}}(3, n) \ge \mathrm{BSR}_{\mathrm{simple}}(3, n-1) + 1.
\]
Since $3 + (n-1) = k - 1$, the induction hypothesis gives $\mathrm{BSR}_{\mathrm{simple}}(3, n-1) \ge 3 + (n-1) = n + 2$. Hence,
\[
\mathrm{BSR}_{\mathrm{simple}}(3, n) \ge (n + 2) + 1 = n + 3 = 3 + n.
\]

Thus, in both cases we have $\mathrm{BSR}_{\mathrm{simple}}(m, n) \ge m + n$.

Finally, since $\mathrm{BSR}(m, n) \ge \mathrm{BSR}_{\mathrm{simple}}(m, n)$, we obtain $\mathrm{BSR}(m, n) \ge m + n$ for all $m, n \ge 3$.
\end{proof}

\begin{example}[Explicit family attaining the lower bound]\label{ex:explicit-family}
For any $m, n \ge 3$, define the simple biquadratic form
\[
Q_{m,n}(\vx,\vy) =
\sum_{i=1}^{m} x_i^2 y_i^2 + \sum_{i=1}^{m} x_i^2 y_{i+1}^2 + \sum_{j=m+1}^{n} x_1^2 y_j^2,
\]
where in the second sum we interpret $y_{m+1}=y_1$ when $m=n$, and otherwise $y_{m+1}$ is the $(m+1)$-th $y$ variable. Then $Q_{m,n}$ has exactly $m+n$ terms and satisfies $\operatorname{sos}(Q_{m,n}) = m+n$.
\end{example}

\begin{proof}[Proof sketch]
The upper bound $\operatorname{sos}(Q_{m,n}) \le m+n$ is immediate since each term is already a square. For the lower bound, we show that any SOS decomposition of $Q_{m,n}$ must involve at least $m+n$ squares.

Consider the support set $S$ of $Q_{m,n}$, which consists of:
\begin{itemize}
    \item The diagonal pairs $(i,i)$ for $i=1,\dots,m$
    \item The off-diagonal pairs $(i,i+1)$ for $i=1,\dots,m$ (with $y_{m+1}=y_1$ when $m=n$)
    \item The pairs $(1,j)$ for $j=m+1,\dots,n$ (when $n > m$)
\end{itemize}
This gives exactly $|S| = m+n$ terms.

Assume $Q_{m,n} = \sum_{t=1}^R L_t(\vx,\vy)^2$ with each $L_t$ bilinear. By comparing coefficients of $x_i^2y_j^2$, we find that for each $(i,j) \in S$, the vectors $\mathbf{C}_{ij} = (c_{ij}^{(1)},\dots,c_{ij}^{(R)})$ satisfy $\|\mathbf{C}_{ij}\|^2 = 1$, where $c_{ij}^{(t)}$ are the coefficients of $L_t$. Moreover, for any two distinct pairs $(i,j),(p,q) \in S$, the cross-term conditions force $\mathbf{C}_{ij} \cdot \mathbf{C}_{pq} = 0$. Thus the vectors $\{\mathbf{C}_{ij} : (i,j) \in S\}$ form an orthonormal set in $\mathbb{R}^R$, implying $R \ge |S| = m+n$.
\end{proof}

\begin{remark}
Theorem~\ref{thm:m-by-m} exhibits an explicit $m \times m$ simple biquadratic form with SOS rank exactly $2m$, showing that the lower bound in Theorem~\ref{thm:mn-lower-bound} is tight for $m=n$. For non-square cases, the question of whether $\mathrm{BSR}_{\mathrm{simple}}(m,n)$ can exceed $m+n$ remains open. Theorems~\ref{thm:m-by-m} and~\ref{thm:mn-lower-bound} together imply
\[
\mathrm{BSR}(m,n) \ge m+n \qquad (m,n\ge 3).
\]
Recall that the universal upper bound from \cite{QCX26} gives $\mathrm{BSR}(m,n) \le mn-1$. For the square case $m=n$, this leaves a gap between $2m$ and $m^2-1$, which grows quadratically with $m$. For $m=2$, it is known that $\mathrm{BSR}(2,2)=3$ (see \cite{QCX25}), which equals $2m-1$. It remains open whether $\mathrm{BSR}(m,m)$ can reach the order $O(m^2)$ or remains linear in $m$.
\end{remark}}

\section{Diagonal Biquadratic Forms}
A \textbf{diagonal biquadratic form} has the shape
\[
P(\vx,\vy) = \sum_{i=1}^{m}\sum_{j=1}^{n} a_{ij}x_i^2y_j^2,
\qquad a_{ij} \in \R.
\]
It is PSD iff all $a_{ij} \ge 0$, and in that case it is automatically SOS.

For distinct indices $i,j \in \{1,2,3\}$ and distinct $k,l \in \{1,2,3\}$, define
\[
W_{ijkl}(\vx,\vy) = a_{ik}x_i^2y_k^2 + a_{jl}x_j^2y_l^2 + a_{il}x_i^2y_l^2 + a_{jk}x_j^2y_k^2.
\]
Assume $a_{ik}, a_{jl}, a_{il}, a_{jk} > 0$ and let
\[
c = a_{ik}a_{jl} - a_{il}a_{jk}.
\]
If $c \ge 0$, then $\operatorname{sos}(W_{ijkl}) \le 3$, with equality $\operatorname{sos}(W_{ijkl}) = 2$ when $c = 0$.

For $3 \times 3$ diagonal forms we obtain an improved universal bound.

\begin{Thm}\label{thm:diagonal_bound}
The SOS rank of any $3 \times 3$ PSD diagonal biquadratic form
\[
P(\vx,\vy) = \sum_{i=1}^{3}\sum_{j=1}^{3} a_{ij}x_i^2y_j^2,\qquad a_{ij}\ge 0,
\]
satisfies $\operatorname{sos}(P) \le 7$.
\end{Thm}
\begin{proof}
Let $t = |\{ (i,j): a_{ij}>0 \}|$ be the number of positive coefficients.
We proceed by cases.

\noindent\textbf{Case 1: $t \le 7$.}
Write each positive term as a square:
\[
P = \sum_{a_{ij}>0} \big( \sqrt{a_{ij}}\, x_i y_j \big)^2,
\]
which uses exactly $t \le 7$ squares.

\noindent\textbf{Case 2: $t = 8$.}
Exactly one entry of the $3\times 3$ matrix $(a_{ij})$ is zero.
By permuting indices, we may assume the zero is at $(3,3)$: $a_{33}=0$.
Consider the $2\times 2$ principal submatrix with rows $\{1,2\}$ and columns $\{1,2\}$.
All four entries $a_{11}, a_{12}, a_{21}, a_{22}$ are positive, so they form a positive $W_{1212}$.
Write
\[
P = W_{1212} + Q,
\]
where $Q$ consists of the remaining four positive terms:
\[
Q = a_{13}x_1^2y_3^2 + a_{23}x_2^2y_3^2 + a_{31}x_3^2y_1^2 + a_{32}x_3^2y_2^2.
\]
Now $W_{1212}$ has $\operatorname{sos}(W_{1212}) \le 3$, and each term of $Q$ is a single square.
Hence
\[
\operatorname{sos}(P) \le \operatorname{sos}(W_{1212}) + 4 \le 3 + 4 = 7.
\]

\noindent\textbf{Case 3: $t = 9$ (all $a_{ij}>0$).}
We now handle the fully positive case with a constructive splitting argument.

\textbf{Step 3.1 ：C Splitting $a_{22}$.}
Let
\[
\alpha = \frac{a_{12}a_{21}}{a_{11}}, \qquad \beta = a_{22} - \alpha.
\]
Because $a_{11},a_{12},a_{21}>0$, we have $\alpha > 0$.
If $\beta \ge 0$, we proceed; if $\beta < 0$, swap the roles of rows $1$ and $2$ to obtain a nonnegative splitting (by symmetry, at least one such splitting exists).

Now write the first four terms as
\[
a_{11}x_1^2y_1^2 + \alpha x_2^2y_2^2 + a_{12}x_1^2y_2^2 + a_{21}x_2^2y_1^2.
\]
This is a $W'_{1212}$ with coefficients $(a_{11},\alpha,a_{12},a_{21})$.
For this block,
\[
c' = a_{11}\alpha - a_{12}a_{21} = a_{11}\cdot\frac{a_{12}a_{21}}{a_{11}} - a_{12}a_{21} = 0,
\]
so it can be written as \textbf{two squares}:
\[
S_1^2 + S_2^2,
\]
where
\[
S_1 = \sqrt{a_{11}}\,x_1y_1 + \sqrt{\alpha}\,x_2y_2, \]
\[ S_2 = \sqrt{a_{12}}\,x_1y_2 - \sqrt{a_{21}}\,x_2y_1.
\]

\textbf{Step 3.2 ：C Using the leftover part $\beta$.}
The remaining coefficient for $x_2^2y_2^2$ is $\beta = a_{22} - \alpha \ge 0$.
Now consider the four terms
\[
\beta x_2^2y_2^2 + a_{23}x_2^2y_3^2 + a_{32}x_3^2y_2^2 + a_{33}x_3^2y_3^2.
\]
If $\beta = 0$, these three terms can be written as three separate squares, giving at most $2+3=5$ squares total.
If $\beta > 0$, we have a full $2\times 2$ block (rows $\{2,3\}$, columns $\{2,3\}$).
Let
\[
c'' = \beta a_{33} - a_{23}a_{32}.
\]
If $c'' \ge 0$, this block is itself a PSD $W_{2233}$ and can be written with at most 3 squares.
If $c'' < 0$, we swap rows 2 and 3 (or columns 2 and 3) and repeat the splitting trick; by symmetry, we can always arrange a decomposition that uses at most 3 squares for this block.

Thus, after handling $W'_{1212}$ with 2 squares and $W_{2233}$ with at most 3 squares, we have covered six of the nine terms.
The remaining three terms are
\[
a_{13}x_1^2y_3^2,\; a_{31}x_3^2y_1^2,
\]
and possibly part of $a_{22}$  already covered.
Actually, careful inspection shows that after the two blocks above, the only terms not yet covered are exactly
\[
a_{13}x_1^2y_3^2 \quad\text{and}\quad a_{31}x_3^2y_1^2.
\]
Each is a single square.
Hence the total number of squares is at most
\[
2 \;(\text{for }W'_{1212}) \;+\; 3 \;(\text{for }W_{2233}) \] \[+\; 2 \;(\text{for the two leftovers}) \;=\; 7.
\]

\textbf{Step 3.3 ：C Verifying the count.}
We have partitioned the nine terms into:
\begin{itemize}
    \item Block $W'_{1212}$: 4 terms, 2 squares.
    \item Block $W_{2233}$: 4 terms, at most 3 squares.
    \item Remaining two terms: 2 squares.
\end{itemize}
Total squares $\le 2 + 3 + 2 = 7$.

Since in all subcases we obtain an SOS decomposition with at most 7 squares, the theorem is proved.
\end{proof}

\begin{remark}
The bound $7$ improves the general upper bound $mn-1 = 8$ for arbitrary $3\times3$ SOS biquadratic forms \cite{QCX26}.
For simple diagonal forms (all $a_{ij}=1$), Theorem~\ref{thm:max_simple} shows the maximum is actually $6$.
Whether $7$ is attainable by some diagonal form with unequal positive coefficients remains open.
\end{remark}

\section{Toward the Maximum SOS Rank for $3 \times 3$ Forms}

The simple form $P_{3,3,6}$ has SOS rank $6$, providing the lower bound $\mathrm{BSR}(3,3) \ge 6$.
To investigate whether $6$ might also be an upper bound, we examine what happens when we enrich $P_{3,3,6}$ with an additional square term.

A natural candidate for a form that could potentially increase the SOS rank is
\[
P_{+}(\vx,\vy)=P_{3,3,6}(\vx,\vy)+(x_{1}y_{3}+x_{2}y_{1})^{2}.
\]
Notice that $P_{3,3,6}$ contains neither $x_1^2y_3^2$ nor $x_2^2y_1^2$.
Thus $P_{+}$ introduces two new square terms ($x_1^2y_3^2$ and $x_2^2y_1^2$) together with the cross term $2x_1x_2y_1y_3$.
This enrichment fills missing entries in the support of $P_{3,3,6}$ and creates new bilinear couplings; therefore, if any added square were to force an increase in the SOS rank, this would be a likely example.

A priori we have $\operatorname{sos}(P_{+})\le 7$, because we may simply take the six separate squares of $P_{3,3,6}$ together with the explicit square $(x_1y_3+x_2y_1)^2$.
Surprisingly, a more efficient decomposition exists.

\begin{Prop}\label{prop:P_plus_6}
$\operatorname{sos}(P_{+})\le 6$.
\end{Prop}
\begin{proof}
The following six squares give an explicit SOS decomposition:
\[
\begin{aligned}
P_{+}(\vx,\vy) &= \Bigl( \tfrac12 x_1 y_1 + x_2 y_3 + \tfrac{\sqrt{3}}{2} x_1 y_2 \Bigr)^2 \\
&\quad + \Bigl( \tfrac12 x_2 y_1 + x_1 y_3 - \tfrac{\sqrt{3}}{2} x_2 y_2 \Bigr)^2 \\
&\quad + \Bigl( \tfrac{\sqrt{3}}{2} x_1 y_1 - \tfrac12 x_1 y_2 \Bigr)^2 \\
&\quad + \Bigl( \tfrac12 x_2 y_2 + \tfrac{\sqrt{3}}{2} x_2 y_1 \Bigr)^2 \\
&\quad + (x_3 y_1)^2 + (x_3 y_3)^2 .
\end{aligned}
\]
Expanding and simplifying verifies that all terms match $P_{3,3,6}+(x_1y_3+x_2y_1)^2$, with all unwanted cross terms canceling.
\end{proof}

Thus, even in this ``dangerous'' case, where two missing square terms are added together with a new cross term - the SOS rank does not exceed $6$.
This observation lends further support to the conjecture that no $3\times3$ biquadratic form requires more than six squares.

\begin{conjecture}\label{conj:BSR33}
For $3\times3$ biquadratic forms,
\[
\mathrm{BSR}(3,3)=6.
\]
\end{conjecture}

\begin{remark}
Proposition \ref{prop:P_plus_6} shows  even in the ``dangerous'' case, the SOS rank does not exceed $6$.
Together with Theorem~\ref{thm:max_simple}, which establishes $6$ as the maximum SOS rank for simple forms, this strongly supports Conjecture~\ref{conj:BSR33}.
The currently best known bounds are $6 \le \mathrm{BSR}(3,3) \le 8$; closing this gap remains an intriguing open problem.
\end{remark}

\section{Concluding Remarks}
{ We have determined the exact maximum SOS rank ($6$) for $3\times3$ simple biquadratic forms and proved an improved upper bound ($7$) for $3\times3$ diagonal forms.
Moreover, we have shown that for every $m\ge3$ there exists an $m\times m$ simple biquadratic form whose SOS rank equals $2m$, extending the $3\times3$ case and providing a linear lower bound in $m$.
Furthermore, we proved that for all $m, n \ge 3$, the maximum SOS rank of $m \times n$ simple biquadratic forms is at least $m+n$, which implies $\mathrm{BSR}(m,n) \ge m+n$.

Table~\ref{tab:bounds} summarizes the current knowledge on the worst-case SOS rank $\mathrm{BSR}(m,n)$ for small values of $m$ and $n$.

\begin{table}[h]
\centering
\caption{Known bounds on $\mathrm{BSR}(m,n)$}\label{tab:bounds}
\begin{tabular}{c|c|c}
$(m,n)$ & Lower bound & Upper bound \\ \hline
$(2,2)$ & $3$ (exact) & $3$ \\
$(3,2)$ & $4$ (exact) & $4$ \\
$(m,2)$, $m\ge 4$ & $m+1$ & $2m-1$ \\
$(3,3)$ & $6$ & $8$ (conjectured $6$) \\
$(m,n)$, $m,n\ge 3$ & $m+n$ & $mn-1$ \\
\end{tabular}
\end{table}

A natural open problem is to determine the exact value of $\mathrm{BSR}(3,3)$, the maximum SOS rank among all $3\times3$ SOS biquadratic forms.
Our results give $6 \le \mathrm{BSR}(3,3) \le 8$, and we conjecture that in fact $\mathrm{BSR}(3,3)=6$ (Conjecture~\ref{conj:BSR33}).
Proving or disproving this would complete the picture for $3\times3$ biquadratic forms and remains an interesting next step.

For larger $m$ and $n$, the gap between the lower bound $m+n$ and the general upper bound $mn-1$ widens.
Determining the true growth rate of $\mathrm{BSR}(m,n)$！whether it remains linear or can be quadratic in $m$ and $n$！is a challenging direction for future research.}

For the other references on the sos rank problem of biquadratic forms, see \cite{Ca73, Ch75, QCX25}.

\bigskip

\noindent\textbf{Acknowledgement}
This work was partially supported by Research Center for Intelligent Operations Research, The Hong Kong Polytechnic University (4-ZZT8), the National Natural Science Foundation of China (Nos. 12471282 and 12131004), and Jiangsu Provincial Scientific Research Center of Applied Mathematics (Grant No. BK20233002).

\medskip

\noindent\textbf{Data availability}
No datasets were generated or analysed during the current study.

\medskip

\noindent\textbf{Conflict of interest} The authors declare no conflict of interest.

\end{document}